\documentclass[10pt]{article}
\usepackage{amsmath, amsthm}\usepackage{enumerate}
\usepackage{amssymb}\usepackage{bbold}
\usepackage{color}
\newtheorem{theorem}{Theorem}
\newtheorem{definition}{Definition}
\newtheorem{exam}{Example}

\newtheorem{prop}{Proposition}
\newtheorem{lem}{Lemma}
\newtheorem{cor}{Corollary}

\DeclareMathOperator{\convo}{\xrightarrow[]{o}}

\DeclareMathOperator{\convn}{\xrightarrow[]{\|\cdot\|}}

\makeatletter
\renewcommand{\subsection}{\@startsection{subsection}{1}
{0pt}{3.25ex plus 1ex minus.2ex}{-1em}{\normalfont\normalsize\bf}}
\makeatother
\begin{document}

\title{{\bf On collectively $\sigma$-Levi sets of operators}}
\date{}
\maketitle
\author{\centering{{Eduard Emelyanov$^{1}$\\ 
\small $1$ Sobolev Institute of Mathematics, Novosibirsk, Russia}

\abstract{A collectively $\sigma$-Levi set of operators
is a generalization of the $\sigma$-Levi operator. 
By use of collective order convergence, we investigate 
relations between collectively $\sigma$-Levi and 
collectively compact sets of operators.


\large

\bigskip

\section{Introduction}

\hspace{3mm}
Several kinds of Levi operators were studied recently in
\cite{AlEG2022,GE2022,ZC2022,E2024}. 
The present paper concerns collective properties of $\sigma$-Levi operators.

In what follows, vector spaces are real and operators are linear. 
The letters $E$ and $F$ stand for vector lattices,
symbols $\text{\rm L}(E,F)$, $\text{\rm L}_{\text{\rm FR}}(E,F)$, and
\text{\rm K}(E,F) for the spaces of linear, finite rank, and compact operators from $E$ to $F$, 
$B_X$ for the closed unit ball of $X$, and $I_X$ for the identity operator in $X$.
We write $y_n\downarrow 0$, whenever 
$y_{n'}\le y_n$ for all $n'\ge n$ and $\inf_E y_n=0$.

Throughout the paper, we say that a sequence $(x_n)$ in $E$ is order convergent to 
$x\in E$ (briefly, $x_n\stackrel{\text{\rm o}}{\to} x$) if there exists a
sequence $(p_n)$ in $E$, $p_n\downarrow 0$ such that 
$|x_n-x|\le p_n$ holds for all $n$. A sequence $(x_n)$ in $E$ is 
order Cauchy if, for some $p_n\downarrow 0$ in $E$, 
$|x_{n'}-x_{n''}|\le p_n$ whenever $n',n''\ge n$. A vector lattice
$E$ is said to be sequentially order complete whenever each 
order Cauchy sequence in $E$ is order convergent.

The following definition is an adopted version of \cite[Definition 1.1]{AlEG2022} and 
\cite[Definition 1]{GE2022}.

\begin{definition}\label{order-to-topology} 
{\em
An operator $T$ from a normed lattice $E$ to a vector lattice $F$ is:
\begin{enumerate}[$a)$]
\item 
\text{\rm $\sigma$-Levi} if, for every 
increasing bounded sequence $(x_n)$ in $E_+$,
there exists $x\in E$ with $Tx_n\stackrel{\text{\rm o}}{\to} Tx$.
The set of such operators is denoted by $\text{\rm L}^\sigma_{\text{\rm Levi}}(E,F)$.
\item 
\text{\rm quasi-c-$\sigma$-Levi} if, for every 
increasing bounded sequence $(x_n)$ in $E_+$,
the sequence $(Tx_n)$ is order convergent.
The set of such operators is denoted by $\text{\rm L}^\sigma_{\text{\rm qcLevi}}(E,F)$.
\item 
\text{\rm quasi-$\sigma$-Levi} if, for every 
increasing boun\-ded sequence $(x_n)$ in $E_+$,
the sequence $(Tx_n)$ is order Cauchy.
The set of such operators is denoted by $\text{\rm L}^\sigma_{\text{\rm qLevi}}(E,F)$.
\end{enumerate}}
\end{definition}
\medskip
\noindent
Clearly, 
$\text{\rm L}^\sigma_{\text{\rm Levi}}(E,F)\subseteq\text{\rm L}^\sigma_{\text{\rm qcLevi}}(E,F)   
\subseteq\text{\rm L}^\sigma_{\text{\rm qLevi}}(E,F)$.
The following example shows that both inclusions are proper in general
(cf., \cite[Example 1]{E2024}).

\begin{exam}\label{quasi KB yet not KB}
{\em
First we show that the inclusion 
$\text{\rm L}^\sigma_{\text{\rm Levi}}(E)\subseteq\text{\rm L}^\sigma_{\text{\rm qcLevi}}(E)$
can be proper.
Define an operator $T$ on $E=C[0,1]\oplus L_1[0,1]$, by
$T\bigl((\phi,\psi)\bigl):=(0,\phi)$ for $\phi\in C[0,1]$ and $\psi\in L_1[0,1]$. 
Clearly, $T\in\text{\rm L}^\sigma_{\text{\rm qcLevi+}}(E)$.
Consider $\phi_n\in C[0,1]$ that equals 
to 1 on $\bigl[0,\frac{1}{2}-\frac{1}{2^n}\bigl]$,
to 0 on $\bigl[\frac{1}{2},1\bigl]$, and is linear otherwise.
Let $f_n:=(\phi_n,0)$. Then $(f_n)$ is bounded and increasing in $E_+$,
and $Tf_n\convo(0,g)$, where $g\in L_1[0,1]$ is the indicator 
function of $\bigl[0,\frac{1}{2}\bigl]$.
Since $g\not\in C[0,1]$, there is no such an $f\in E$ that
$Tf=(0,g)$, and hence $T\notin\text{\rm L}^\sigma_{\text{\rm Levi}}(E)$.

For the second inclusion, consider the Banach lattice $c$ of convergent real sequences
and denote elements of $c$ by $\sum_{n=1}^\infty a_n \cdot e_n$,
where $e_n$ is the n-th unit vector of $c$ and $(a_n)$ converges in $\mathbb{R}$.
Since each bounded increasing 
sequence in $c_+$ is \text{\rm o}-Cauchy, $I_c\in\text{\rm L}^\sigma_{\text{\rm qLevi}}(c)$.
However, a bounded increasing sequence $I_cf_n=f_n :=\sum_{k=1}^n e_{2k-1}$ in $c_+$ is not 
order convergent. Thus, $I_c\notin\text{\rm L}^\sigma_{\text{\rm qcLevi}}(c)$.
}
\end{exam}

\medskip
We shall use the following lemma (cf.,  \cite[Lemmas 1,2]{E2024}).

\begin{lem}\label{prop2}
Let $E$ be a normed lattice and let $F$ be a vector lattice. The following holds.
\begin{enumerate}[$i)$]
\item
$\text{\rm L}^\sigma_{\text{\rm qcLevi}}(E,F)$,
and $\text{\rm L}^\sigma_{\text{\rm qLevi}}(E,F)$ are vector spaces.
\item
$\text{\rm L}_{\text{\rm FR}}(E,F)\subseteq\text{\rm L}^\sigma_{\text{\rm Levi}}(E,F)$.
\item
If $F$ is a normed lattice then 
$\text{\rm K}_+(E,F)\subseteq\text{\rm L}^\sigma_{\text{\rm qcLevi}}(E,F)$.
\end{enumerate}
\end{lem}

\begin{proof}
$i)$ It is trivial.  

\medskip
$ii)$\
Let $T\in\text{\rm L}_{\text{\rm FR}}(E,F)$, say
$T = \sum_{k=1}^n f_k \otimes y_k$ 
for $y_1, \dots, y_n \in T(E)$ and $f_1, \dots, f_n \in E'$.
Denote
$$
   T_1:=\sum_{k=1}^n f_k^+ \otimes y_k \ \ \text{\rm and} \ \ 
   T_2:=\sum_{k=1}^n f_k^- \otimes y_k. 
$$
Let $(x_m)$ be an increasing bounded sequence in $E_+$. 
Then, for each $k$, the sequences $f_k^+(x_m)$ 
and $f_k^-(x_m)$ are increasing and bounded. 
Thus, $f_k^+(x_m)\to a_k$ and 
$f_k^-(x_m)\to b_k$ for some $a_k,b_k\in\mathbb{R}_+$. Since
$\dim(T(E))<\infty$,
$$
   T_1x_m\convo\sum_{k=1}^n a_k y_k\in T(E)\ \ \ \text{\rm and} \ \ \ 
   T_2x_m\convo\sum_{k=1}^n b_k y_k\in T(E).
$$
Therefore,
$$
  Tx_m=(T_1x_m-T_2x_m)\convo\sum_{k=1}^n(a_k-b_k)y_k\in T(E).
$$
Take an $x\in E$ such that $Tx=\sum_{k=1}^n(a_k-b_k)y_k$. Then $Tx_m\convo Tx$.
We conclude $T\in\text{\rm L}^\sigma_{\text{\rm Levi}}(E,F)$.

\medskip
$iii)$\ 
Let $T\in\text{\rm K}_+(E,F)$ and let $(x_m)$ be an 
increasing sequence in $(B_E)_+$. Then $(Tx_m)$ has a subsequence 
$(Tx_{m_j})$ satisfying $\left\|Tx_{m_j}-y\right\|\to 0$ 
for some $y \in F$. Since $Tx_m\uparrow$ then $\left\|Tx_m-y\right\|\to 0$.
As each norm convergent increasing sequence converges 
in order to the same limit then $Tx_m\convo y$, and consequently 
$T\in\text{\rm L}^\sigma_{\text{\rm qcLevi}}(E,F)$.
\end{proof}

The next example strengthens Example~\ref{quasi KB yet not KB}  
by showing that the inclusion 
$
   \text{\rm L}^\sigma_{\text{\rm Levi}}(E)\cap\text{\rm K}_+(E)\subseteq
   \text{\rm L}^\sigma_{\text{\rm qcLevi}}(E)\cap\text{\rm K}_+(E)
$
can be proper (cf., \cite[Example 2]{E2024}).

\begin{exam}\label{Example1 c_0}
{\em
Let $(\alpha_n)$ be a vanishing real sequence consisting of non-zero positive terms.
Define an operator $S$ from $c$ to $c_0$ by
$
   S\left( \sum_{n=1}^\infty a_n e_n\right)  = \sum_{n=1}^\infty(\alpha_n a_n) e_n. 
$
Then $S\in\text{\rm K}_+(c,c_0)$, and hence 
$S\in\text{\rm L}^\sigma_{\text{\rm qcLevi}}(c,c_0)$
by Lemma~\ref{prop2}. 
Take a bounded increasing sequence 
$x_n=\sum\limits_{k=1}^n e_{2k}$ in $c_+$
The sequence 
$(Sx_n) =\Bigl(\sum\limits_{k=1}^n\alpha_{2k} e_{2k}\Bigl)$ converges in order to 
$\sum\limits_{k=1}^\infty\alpha_{2k} e_{2k}\in c_0$, 
however there is no $x\in c$ with
$Sx=\sum\limits_{k=1}^\infty\alpha_{2k} e_{2k}$.
Indeed, would such $x=\sum\limits_{k=1}^\infty a_k e_k\in c$ with
$Sx=S\left(\sum\limits_{k=1}^\infty a_k e_k \right)=
\sum\limits_{k=1}^\infty\alpha_{2k} e_{2k}$ exist, it must satisfies $a_k=1$ 
for even $k$-th and $a_k=0$ 
for odd $k$-th, which is absurd. 
Therefore, $S\notin\text{\rm L}^\sigma_{\text{\rm Levi}}(c,c_0)$.

The operator $S$ is also a counter-example to \cite[Prop.3.5]{AlEG2022}.

Now, define a sequence $(S_i)$ of operators in $\text{\rm L}_{\text{\rm FR}}(c,c_0)$ by
$
   S_i\left(\sum_{n=1}^\infty a_n e_n\right)  = \sum_{n=1}^i(\alpha_n a_n) e_n. 
$
Trivially, $S_i\convn S$. By Lemma~\ref{prop2},
$S_i\in\text{\rm L}^\sigma_{\text{\rm Levi}}(c,c_0)$.
Since $S\notin\text{\rm L}^\sigma_{\text{\rm Levi}}(c,c_0)$ then
the set $\text{\rm L}^\sigma_{\text{\rm Levi}}(c,c_0)$ is not closed
under the operator norm.
}
\end{exam}

It is worth noting that, generally, $\text{\rm L}^\sigma_{\text{\rm Levi}}(E,F)$ need not 
to be a vector space \cite[Example 8]{E2024}.

\medskip
The present paper is organized as follows.
Section 2 is devoted to elementary properties of collective order convergence of
families of sequences. In Section 3 is devoted to collectively $\sigma$-Levi sets of operators,
their relations to collectively compact sets, and for the domination problem.

For unexplained terminology and notation we refer to \cite{AB2003,AP1968,Kus2000,Mey1991}.

\newpage

\section{Collective order convergence}

\hspace{3mm}
Working with families of sequences of elements of a vector lattice requires
a certain notion of ``collective" order convergence. In what follows, we
identify $E$-valued sequences and elements of the vector lattice $E^{\mathbb{N}}$ 
equipped with the pointwise linear and lattice operations.  

\begin{definition}\label{COCS} 
{\em
Let $\mathcal{A}\subseteq E^{\mathbb{N}}$.
We say that $\mathcal{A}$ collective order converges to an indexed 
subset $\{c_a\}_{a\in\mathcal{A}}$ of $E$
(briefly, $\mathcal{A}\stackrel{\text{\rm c-o}}{\to}\{c_a\}_{a\in\mathcal{A}}$)
whenever there exists a sequence $(p_n)$ in $E$, $p_n\downarrow 0$ 
such that $|a_n-c_a|\le p_n$
holds for all $n$ and all $(a_n)\in\mathcal{A}$.
We call $\mathcal{A}$ collective order-null if 
$\mathcal{A}\stackrel{\text{\rm c-o}}{\to}\{0\}_{a\in\mathcal{A}}$
}
\end{definition}

In this section we give some elementary properties of 
collective order convergence which are used in Section 3.
The following proposition is elementary and its proof is left to the reader.

\begin{prop}\label{elem prop}
Let $\mathcal{A}$ and $\mathcal{B}$ be 
nonempty collective order convergent subsets of $E^{\mathbb{N}}$,
and let $\alpha,\beta\in\mathbb{R}$.
The following sets are collective order convergent.
\begin{enumerate}[$i)$]
\item
$\mathcal{A}\cup\mathcal{B}$.
\item
$\alpha\mathcal{A}+\beta\mathcal{B}:=
\{(\alpha a_n+\beta b_n)\}_{a\in\mathcal{A}; b\in\mathcal{B}}$.
\item
$|\mathcal{A}|:=\{(|a_n|)\}_{a\in\mathcal{A}}$.
\item
The convex hull $\text{\rm co}(\mathcal{A})$ of $\mathcal{A}$ in $E^{\mathbb{N}}$.
\end{enumerate}
Moreover,
\begin{enumerate}
\item[$v)$]
If $\mathcal{A}\stackrel{\text{\rm c-o}}{\to}\{c_a\}_{a\in\mathcal{A}}$ and
$\mathcal{A}\stackrel{\text{\rm c-o}}{\to}\{c'_a\}_{a\in\mathcal{A}}$ 
then $c_a=c'_a$ for all $a\in\mathcal{A}$.
\item[$vi)$]
$\mathcal{A}\stackrel{\text{\rm c-o}}{\to}\{c_a\}_{a\in\mathcal{A}}$ iff
$\{(a_n-c_a)\}_{a\in\mathcal{A}}\stackrel{\text{\rm c-o}}{\to}\{0\}_{a\in\mathcal{A}}$.
\item[$vii)$]
A sequence $(a_n)$ in $E$ order converges iff the set $\{(a_n)\}$ is collective order convergent.
\end{enumerate}
\end{prop}
\noindent
Note that the passing to solid hull does not preserve collective order convergence
for any nontrivial $E$. Indeed, let $0\ne x\in E$. Then the set 
$\mathcal{A}=\{(a_n): a_n\equiv x\}$ of one constantly $x$ sequence 
is collective order convergent,
yet its solid hull $\text{\rm sol}(\mathcal{A})$ is not, as
$\text{\rm sol}(\mathcal{A})$ contains a sequence 
$\Bigl(\frac{1+(-1)^n}{2}x\Bigl)$ that does not order converge.
It should be clear that if $E$ is an Archimedean vector lattice then,
for each order convergent to zero sequence $(a_n)$ in $E$ possessing 
at least one non-zero term, the set $\{(\lambda a_n): \lambda\in\mathbb{R}\}$
is not collective order-null. Also, it is worth noting that the set 
$\{(\delta_k^n)_n: k\in\mathbb{N}\}$ consisting of order-null real sequences
is not collective order-null.

\medskip
The next theorem extends items $ii)$ and $iv)$ of Proposition~\ref{elem prop}
to the Banach lattice setting as follows.

\begin{theorem}\label{elem BL prop}
Let $E$ be a Banach lattice, let $(p_{i,n})_n$ be sequences in $B_E$
satisfying $p_{i,n}\downarrow 0$ for each $i\in\mathbb{N}$, and 
let $\mathcal{A}_i$ be nonempty subsets of $E^\mathbb{N}$
such that $|a_{i,n}|\le p_{i,n}$ holds for all $i,n\in\mathbb{N}$, and 
$ (a_{i,n})_n\in\mathcal{A}_i$. Then, the set
$$
   \sum\limits_{i=1}^\infty \alpha_i \mathcal{A}_i=
   \Big\{\Big(\sum\limits_{i=1}^\infty \alpha_i a_{i,n}\Big): 
   (a_{i,n})_n\in\mathcal{A}_i\ \text{\rm and}\ \sum\limits_{i=1}^\infty|\alpha_i|\le 1\Big\}
$$
is collective order-null. In particular, for every $M>0$ and 
$\mathcal{A}\stackrel{\text{\rm c-o}}{\to}\{0\}_{a\in\mathcal{A}}$
in $E^{\mathbb{N}}$, the set
$$
   \Big\{\Big(\sum\limits_{i=1}^\infty\alpha_ia_{i,n}\Big): (a_{i,n})_n\in\mathcal{A}
   \ \text{\rm and} \ \sum\limits_{i=1}^\infty|\alpha_i|\le M\Big\}
$$
is collective order-null.
\end{theorem}

\begin{proof}
Passing to the norm-limit as $m\to\infty$ in the following inequality
$$
  \Big|\sum\limits_{i=1}^m \alpha_i a_{i,n}\Big|\le
   \sum\limits_{i=1}^m |\alpha_i| |a_{i,n}|\le \sum\limits_{i=1}^m |\alpha_i| p_{i,n}
  \le\sum\limits_{i=1}^\infty |\alpha_i| p_{i,n},
$$
where $(a_{i,n})_n\in\mathcal{A}_i$, we obtain 
$\Big|\sum\limits_{i=1}^\infty \alpha_i a_{i,n}\Big|\le
p_n:=\sum\limits_{i=1}^\infty |\alpha_i| p_{i,n}$ for all $n$. 
Clearly, $(p_n)$ is decreasing. It remains to proof $p_n\downarrow 0$.
Suppose in contrary $0<a\le p_n$ for all $n$. Fix an arbitrary $m\in\mathbb{N}$. 
Since $0<a\le\sum\limits_{i=1}^m |\alpha_i| p_{i,n}+
\sum\limits_{i=m+1}^\infty |\alpha_i| p_{i,n}$ for all $n$, and since 
$\inf\limits_{n\in\mathbb{N}}\sum\limits_{i=1}^m |\alpha_i| p_{i,n}=0$, 
we obtain that $0<a\le\sum\limits_{i=m+1}^\infty |\alpha_i| p_{i,n}$
for all $m,n\in\mathbb{N}$. Therefore, 
$$
   0<\|a\|\le\limsup\limits_{m\to\infty}\Big\|\sum\limits_{i=m+1}^\infty |\alpha_i| p_{i,n}\Big\|\le
   \lim\limits_{m\to\infty}\sum\limits_{i=m+1}^\infty |\alpha_i|=0,
$$
which is absurd.

The rest of proof follows from the previous part by taking
$\mathcal{A}_i=\mathcal{A}$ for all $i\in\mathbb{N}$.
\end{proof}

We finish this section with the following notion of collective order Cauchy
set of sequences.

\begin{definition}\label{COCompleteS} 
{\em
A set $\mathcal{A}\subseteq E^{\mathbb{N}}$
is collective order Cauchy if, for some $p_n\downarrow 0$ in $E$, 
$|a_{n'}-a_{n''}|\le p_n$ holds for all $a\in\mathcal{A}$ whenever $n',n''\ge n$.
A vector lattice $E$ is sequentially collective order complete if each collective order Cauchy
subset of $E^{\mathbb{N}}$ is collective order convergent.
}
\end{definition}

\noindent
The next elementary proposition shows that the sequential collective order completeness
agrees with sequential order completeness.

\noindent
\begin{prop}\label{elem complete}
Let $E$ be a vector lattice. The following conditions
are equivalent.
\begin{enumerate}[$i)$]
\item
If a sequence $(x_n)$ in $E$ satisfies $|x_{n'}-x_{n''}|\le p_n$, whenever $n',n''\ge n$
for some $p_n\downarrow 0$ in $E$, then there exists $x\in E$ with  
$|x_n-x|\le p_n$ for all $n$.
\item
If a subset $\mathcal{A}$ of $E^{\mathbb{N}}$
satisfies $|a_{n'}-a_{n''}|\le p_n$ for all $a\in\mathcal{A}$
and some $p_n\downarrow 0$ in $E$, whenever $n',n''\ge n$,
then there exists an indexed subset $\{c_a\}_{a\in\mathcal{A}}$ of $E$
such that $|a_n-c_a|\le p_n$ holds for all $n$ and all $(a_n)\in\mathcal{A}$.
\item 
$E$ is sequentially collective order complete.
\item 
$E$ is sequentially order complete.
\end{enumerate}
\end{prop}

\begin{proof}
Implications $i)\Longrightarrow ii)\Longrightarrow iii)\Longrightarrow iv)$ are trivial.

\medskip
$iv)\Longrightarrow i)$\
Let $|x_{n'}-x_{n''}|\le p_n$ for all $n',n''\ge n$ and some $p_n\downarrow 0$ in $E$.
Since $E$ is sequentially order complete, $x_n\convo x$ for some $x\in E$.
Sending $n''\to\infty$ and passing to the order limit in the inequality $|x_n-x_{n''}|\le p_n$, 
where $n''\ge n$, we obtain $|x_n-x|\le p_n$ for all $n$.
\end{proof}

\section{Collectively $\sigma$-Levi sets of operators}

\hspace{3mm}
Recall that a set $A$ of operators between normed spaces $X$ and $Y$
is collectively compact whenever 
$\bigcup\limits_{T\in A}T(B_X)$ is relatively compact in $Y$ \cite{AP1968}.
This section is devoted to collectively $\sigma$-Levi sets of operators,
their relation to collectively compact sets, and the domination problem
for collectively $\sigma$-Levi sets. We begin with 
the following collective version of Definition~\ref{order-to-topology}.

\begin{definition}\label{collective Levi} 
{\em
Let $E$ be a normed lattice, $F$ a vector lattice, and $A\subseteq\text{\rm L}(E,F)$.
We say that $A$ is:
\begin{enumerate}[$a)$]
\item 
a \text{\rm collectively $\sigma$-Levi set} if, for every 
increasing bounded $(x_n)$ in $E_+$,
there exists an indexed subset $\{x_T\}_{T\in A}$ of $E$
satisfying $\{(Tx_n): T\in A\}\stackrel{\text{\rm c-o}}{\to}\{Tx_T\}_{T\in A}$.
\item 
a \text{\rm collectively quasi-c-$\sigma$-Levi} set if, for every 
increasing bounded $(x_n)$ in $E_+$,
there exists an indexed subset $\{y_T\}_{T\in A}$ of $F$
satisfying $\{(Tx_n): T\in A\}\stackrel{\text{\rm c-o}}{\to}\{y_T\}_{T\in A}$.
\item 
a \text{\rm collectively quasi-$\sigma$-Levi} set if, for every 
increasing bounded $(x_n)$ in $E_+$,
the set $\{(Tx_n): T\in A\}\subseteq F^{\mathbb{N}}$ is
collective order Cauchy.
\end{enumerate}}
\end{definition}
\medskip
\noindent
Obviously, $T$ lies in $\text{\rm L}^\sigma_{\text{\rm Levi}}(E,F)$
($\text{\rm L}^\sigma_{\text{\rm qcLevi}}(E,F)$, 
$\text{\rm L}^\sigma_{\text{\rm qLevi}}(E,F)$)
iff the set $\{T\}$ is a collectively $\sigma$-Levi
(resp., collectively quasi-c-$\sigma$-Levi, collectively quasi-$\sigma$-Levi)
subset of $\text{\rm L}(E,F)$.

\medskip
We continue with the question on which properties of $\sigma$-Levi,
quasi-c-$\sigma$-Levi, and quasi-$\sigma$-Levi operators
mentioned in Lemma~\ref{prop2} have collective versions.
The properties described in Lemma~\ref{prop2}\,$i)$
have the following extension.

\begin{prop}\label{prop2-CLS}
Let $E$ be a normed lattice, $F$ a vector lattice, and $A,B\subseteq\text{\rm L}(E,F)$. 
The following holds.
\begin{enumerate}[$i)$]
\item
If $A$ and $B$ are both collectively quasi-c-$\sigma$-Levi then the set
$\{\alpha T+\beta S: |\alpha|+|\beta|\le 1, T\in A, S\in B\}$ is also
collectively quasi-c-$\sigma$-Levi.
\item
If $A$ and $B$ are both collectively quasi-$\sigma$-Levi then the set
$\{\alpha T+\beta S: |\alpha|+|\beta|\le 1, T\in A, S\in B\}$ is also
collectively quasi-$\sigma$-Levi.
\end{enumerate}
\end{prop}

\begin{proof}
$i)$\ 
By the assumption, there exist sequences $p_n\downarrow 0$, $q_n\downarrow 0$ in $F$,
and indexed subsets $\{y_T\}_{T\in A}$, $\{z_S\}_{S\in B}$ of $F$ satisfying 
$|Tx_n-y_T|\le p_n$, $|Sx_n-z_S|\le q_n$ for all $T\in A$, $S\in B$, and $n\in\mathbb{N}$.
The result follows from 
$$
   |(\alpha T+\beta S)x_n-(\alpha y_T+\beta z_S)|\le 
   |\alpha|p_n+|\beta|q_n\le (p_n+q_n)\downarrow 0.
$$
$ii)$\ 
Let sequences $(p_n)$, $(q_n)$ in $F$
satisfy $|Tx_{n'}-Tx_{n''}|\le p_n\downarrow 0$ and $|Sx_{n'}-Sx_{n''}|\le q_n\downarrow 0$
for all $T\in A$, $S\in B$, and $n',n''\ge n$. The result follows from 
$$
   |(\alpha T+\beta S)x_{n'}-(\alpha T+\beta S)x_{n''})|\le 
   (p_n+q_n)\downarrow 0,
$$
for $n',n''\ge n$.
\end{proof}

The items $ii)$ and $iii)$ of Lemma~\ref{prop2} have no reasonable 
collective extension. To see this, define norm-one functionals 
$T_k$ on $c_0$ by $T_ka=a_k$.
Thus, $T_k\in\text{\rm L}_{\text{\rm FR}}(c_0,\mathbb{R})$,
yet the set $\{T_k\}_{k\in\mathbb{N}}$ is not even collectively quasi-$\sigma$-Levi.
Indeed, for the increasing bounded sequence $x_n=\sum_{m=1}^n e_m$
in $c_0$, there is no sequence $p_n\downarrow 0$ in $\mathbb{R}$
with $|T_kx_{n'}-T_kx_{n''}|\le p_n$ for all $k$ and $n',n''\ge n$,
since $|T_{n+1}x_n-T_{n+1}x_{n+1}|=1$ for every $n$.
Moreover, $\{T_k\}_{k\in\mathbb{N}}$ is a collectively compact subset 
of ${\rm L}_+(c_0,\mathbb{R})$ that is not collectively quasi-$\sigma$-Levi.

\medskip
Now, we apply Theorem~\ref{elem BL prop}
for strengthening Proposition~\ref{prop2-CLS}
in the Banach lattice setting as follows.

\begin{theorem}\label{prop2iii-CLS}
Let $E$ be a normed lattice, $F$ a Banach lattice, and let
$A$ be a bounded collectively quasi-c-$\sigma$-Levi subset of $\text{\rm L}(E,F)$.
Then the set
$$
   \Big\{\Big(\sum\limits_{i=1}^\infty\alpha_iT_i\Big): T_i\in A
   \ \text{\rm and} \ \sum\limits_{i=1}^\infty|\alpha_i|\le 1\Big\}
$$
is collectively quasi-c-$\sigma$-Levi, where 
$\sum\limits_{i=1}^\infty\alpha_iT_i$ is the limit of 
partial sums $\sum\limits_{i=1}^n\alpha_iT_i$ in the operator norm.
\end{theorem}

\begin{proof}
Let $(x_n)\uparrow$ in $(B_E)_+$. Then 
$\{(Tx_n)\}_{T\in A}\stackrel{\text{\rm c-o}}{\to}\{y_T\}_{T\in A}$
for some subset $\{y_T\}_{T\in A}$ of $F$.
Proposition~\ref{elem prop}\,$vi)$ gives 
$\{(Tx_n-y_T)\}_{T\in A}\stackrel{\text{\rm c-o}}{\to}\{0\}_{T\in A}$.
By Theorem~\ref{elem BL prop}, 
$$
   \Big\{\Big(\sum\limits_{i=1}^\infty\alpha_i(T_ix_n-y_{T_i})\Big)\Big\}_{T_i\in A,
   \sum\limits_{i=1}^\infty|\alpha_i|\le 1}
   \stackrel{\text{\rm c-o}}{\to}\{0\}_{T_i\in A,
   \sum\limits_{i=1}^\infty|\alpha_i|\le 1}.
$$
So, there exists a sequence $p_n\downarrow 0$ in $F$ satisfying 
$$
   \Big|\Big(\sum\limits_{i=1}^\infty\alpha_iT_i\Big)x_n-
   \sum\limits_{i=1}^\infty\alpha_iy_{T_i}\Big|=
   \Big|\sum\limits_{i=1}^\infty\alpha_i(T_ix_n-y_{T_i})\Big|\le p_n
$$
for all $n$, $T_i\in A$, and all $\alpha_i$ with 
$\sum\limits_{i=1}^\infty|\alpha_i|\le 1$,
where the series $\sum\limits_{i=1}^\infty\alpha_iT_i$
converges in the operator norm due to boundedness of $A$.
The proof is complete.
\end{proof}
\noindent 
Since every Dedekind $\sigma$-complete vector lattice is sequentially order complete,
the next corollary follows from Proposition \ref{elem complete}
and Theorem \ref{prop2iii-CLS}.

\begin{cor}\label{cor2iii-CLS}
Let $E$ be a normed lattice, $F$ be a Dedekind $\sigma$-complete Banach lattice, and
$A$ be a bounded collectively quasi-$\sigma$-Levi subset of $\text{\rm L}(E,F)$. Then
the set 
$
   \Big\{\Big(\sum\limits_{i=1}^\infty\alpha_iT_i\Big): T_i\in A
   \ \text{\rm and} \ \sum\limits_{i=1}^\infty|\alpha_i|\le 1\Big\}
$
is collectively quasi-c-$\sigma$-Levi.
\end{cor}

\medskip
Now, we discuss of the ``collective" domination problem 
for Levi sets of operators. First, recall some already known 
related results for Levi operators. 

The quasi-$\sigma$-Levi operators do satisfy the domination property
(cf. \cite[Theorem 2.7]{AlEG2022}, \cite[Theorem 3]{GE2022}). 
We do not know where or not the quasi-c-$\sigma$-Levi operators 
satisfy the domination property. In general, $\sigma$-Levi operators 
do not satisfy the domination property (cf. \cite[Example 7]{E2024}).

\begin{exam}\label{domination for sigma Levi fails}
{\em
Define operators $S,T\in\text{\rm L}(c)$ by 
$$
   S\left(\sum_{n=1}^\infty a_n e_n\right)=\sum_{n=1}^\infty \frac{a_n}{2^n} e_n;
   T\left(\sum_{n=1}^\infty a_n e_n\right)=
   \sum_{n=1}^\infty\Bigl(\sum_{k=1}^\infty\frac{a_k}{2^k}\Bigl) e_n.
$$
Then $0\le S\le T$. Operator $T$ has rank one, and hence  
$T$ is $\sigma$-Levi by Lemma \ref{prop2}~$ii)$. 
However, $S\notin\text{\rm L}^\sigma_{\text{\rm Levi}}(c)$ 
due to Example \ref{Example1 c_0}.}
\end{exam}

\medskip
We use the following ``collective" notion of domination
for sets of operators. 

\begin{definition}\label{collective domination}
{\em
Let $A,B\subseteq\text{\rm L}_+(E,F)$.
Then $A$ is dominated by $B$ if, for each $S\in A$,
there exists $T\in B$ with $S\le T$.}
\end{definition}

\noindent
We  conclude the paper with the following ``collective" partial generalization of
\cite[Theorem~2.7]{AlEG2022} in the class of quasi-$\sigma$-Levi operators.

\begin{theorem}\label{col-do-qLevi}
{\em 
Let $E$ be a normed lattice, $F$ be a vector lattice, 
and $A,C\subseteq\text{\rm L}_+(E,F)$ be such that
$A$ is dominated by $C$. If $C$ is collectively quasi-$\sigma$-Levi
then $A$ is also collectively quasi-$\sigma$-Levi.}
\end{theorem}

\begin{proof}
Let $(x_n)$ be an increasing sequence in $(B_E)_+$.
By the assumption, $C$ is collectively quasi-$\sigma$-Levi,
and hence the set $\{(Tx_n): T\in C\}\subseteq F^{\mathbb{N}}$ is
collective order Cauchy. By Definition \ref{COCompleteS},
for some $p_n\downarrow 0$ in $F$, 
$|Tx_{n'}-Tx_{n''}|\le p_n$ holds for all $T\in C$ whenever $n',n''\ge n$.

Let $S\in A$. Then $0\le S\le T_S$ for some $T_S\in C$.
Since $|T_Sx_{n'}-T_Sx_{n''}|\le p_n$ for $n',n''\ge n$,
$$
   |Sx_{n'}-Sx_{n''}|\le|Sx_{n'}-Sx_n|+|Sx_{n''}-Sx_n|=
$$
$$
   S(x_{n'}-x_n)+S(x_{n''}-x_n)\le T_S(x_{n'}-x_n)+T_S(x_{n''}-x_n)= 
$$
$$
   |T_Sx_{n'}-T_Sx_{n}|+|T_Sx_{n''}-T_Sx_{n}|\le 2p_n 
$$
for all $n',n''\ge n$. Because $S\in A$ is arbitrary and $2p_n\downarrow 0$, we conclude
that $A$ is collectively quasi-$\sigma$-Levi.
\end{proof}

{
}
\end{document}